\documentclass[12pt,a4paper]{amsart}
\calclayout
\numberwithin{equation}{section}
\allowdisplaybreaks
\theoremstyle{plain}

\allowdisplaybreaks
\usepackage{enumerate}
\usepackage{mathtools}
\mathtoolsset{showonlyrefs}
\usepackage{mathrsfs}
\usepackage{graphicx}
\usepackage{subfigure}

\newtheorem{theorem}{Theorem}[section]
\newtheorem{lemma}[theorem]{Lemma}

\numberwithin{equation}{section}
\newtheorem{remark}{Remark}

\newcommand{\R}{\mathbb{R} }
\newcommand{\E}{\mathbb{E} }
\newcommand{\D}{\mathcal{D}}
\newcommand{\F}{\mathcal{F}}
\newcommand{\Z}{\widetilde{Z} }

\begin{document}
	\title[Branching Brownian motion with absorption]{An ergodic theorem for the maximum of branching Brownian motion with absorption}
	\author[F. Yang]{Fan Yang}
	\address{Fan Yang\\ School of Mathematical Sciences \\  Beijing University of Posts and Telecommunications\\ Beijing 100876\\  China}
	\email{fan-yang@bupt.edu.cn}
	\thanks{The research of this project is supported by the National Key R\&D Program of China (No. 2020YFA0712900). The research of F. Yang is supported by China Postdoctoral Science Foundation (No. 2023TQ0033) and Postdoctoral Fellowship Program of CPSF (No. GZB20230068).}

	\begin{abstract} 
		In this paper, we study branching Brownian motion with absorption, in which particles undergo Brownian motions and are killed upon hitting the absorption barrier. We prove that the empirical distribution function of the maximum of this process converges almost surely to a randomly shifted Gumbel distribution.
	\end{abstract}
	
	\maketitle

	\section{Introduction}\label{Sec:intro}
	A classical branching Brownian motion (BBM) in $\R$ can be constructed as follows. Initially there is a single particle at the origin of the real line and this particle moves as a 1-dimensional standard Brownian motion denoted by $B = \{B(t), t\geq 0 \}$. 
	After an independent exponential time with parameter $1$, the initial particle dies and gives birth to $L$ offspring, where $L$ is a positive integer-valued random variable with distribution $\{p_k: k\geq 1 \}$. Here we assume that the expected number of offspring is $2$ (i.e., $\sum_{k=1}^{\infty} kp_k = 2$) and the variance of the offspring distribution is finite (i.e., $\sum_{k=1}^{\infty} k(k-1)p_k < \infty$). Each offspring starts from its creation position and evolves independently, according to the same law as its parent. We denote the collection of particles alive at time $t$ as $N_t$. For any $u\in N_t$ and $s\le t$, let $X_u(s)$ be the position at time $s$ of particle $u$ or its ancestor alive at that time. 
	
	McKean \cite{McKean75} established a connection between BBM and the Fisher-Kolmogorov-Petrovskii-Piskounov (F-KPP) equation
	\begin{equation}\label{KPPeq1}
		\frac{\partial u}{\partial t} = \frac{1}{2} \frac{\partial^2 u}{\partial x^2} + \sum_{k=1}^{\infty} p_k u^k - u.
	\end{equation}
	The F-KPP equation has received entensive attention from both analytic techniques (see, for instance, Kolmogorov et al. \cite{Kolmogorov37} and Fisher \cite{Fisher37}) and probabilistic methods
	(see, for example, McKean \cite{McKean75}, Bramson \cite{Bramson78,Bramson83}, Harris \cite{Harris99} and Kyprianou \cite{Kyprianou04}). 
	
	Let's recall some classical results on BBM and F-KPP equation. Define
	\begin{equation}
		\mathbf{M}_t := \max\{X_u(t): u\in N_t \}.
	\end{equation}
	Bramson \cite{Bramson78} established that
	\begin{equation}
		\lim_{t\to\infty}\mathbb{P}(\mathbf{M}_t - m_t \le z)=\lim_{t\to \infty} u(t,m_t+z)=w(z),\quad z\in \R,
	\end{equation}
	where $m_t:= \sqrt{2} t-\frac{3}{2\sqrt{2}}\log t$ and $w$ solves the ordinary differential equation $\frac{1}{2}w''+\sqrt{2}w'+\sum_{k=1}^{\infty} p_k w^k - w=0$.
	Such a solution $w$ is known as the traveling wave solution.
	Lalley and Sellke \cite{Lalley87} provided the following representation of $w$ for dyadic BBM
	\begin{equation}\label{travelling}
		w(z):=\E\left[e^{-C_* e^{-\sqrt{2} z}Z_{\infty}}\right],
	\end{equation}
	where $C_*$ is a positive constant and $Z_{\infty}$ is the limit of the derivative martingale of BBM. 
	Specifically, define 
	\begin{equation}\label{def_derivative}
		Z_t = \sum_{u\in N_t} (\sqrt{2}t - X_u(t)) e^{\sqrt{2}(X_u(t)-\sqrt{2}t) },
	\end{equation}
	then $Z_t$ serves as the derivative martingale of the BBM. We denote $Z_{\infty}$ as the limit of $Z_t$ $\mathbb{P}_x$-almost surely, as established by Lalley and Sellke \cite{Lalley87} or Kyprianou \cite{Kyprianou04}. In the same paper \cite{Lalley87}, they conjectured that the empirical (time-averaged) distribution of maximal displacement converges almost surely, that is,
	\begin{equation}\label{ergodic_BBM}
		\lim_{T\to\infty} \frac{1}{T} \int_0^T \mathbf{1}_{\{\mathbf{M}_s - m_s \leq x\}} \mathrm{d} s = \exp\left\{-C_* Z_{\infty} e^{-\sqrt{2}x} \right\}, \quad \mbox{a.s.}
	\end{equation}
	This conjecture was later confirmed by Arguin, Bovier and Kistler \cite{ABK13b}.
	
	In this paper, we consider the similar problems of BBM with absorption, where the particle is killed when it hits the absorbing barrier. The process can be defined as follows. Initially there is a single particle at $x>0$ and this particle evolves as the classical BBM with branching rate $1$. We also assume that the number of offspring $L$ has distribution $\{p_k,k\geq 1\}$ with $\E L = 2$ and $\E L^2 < \infty$. In addition, we add an absorbing barrier at the line $\{(y,t):y= \rho t \}$ for some $\rho\in\R$, i.e. particles hitting the barrier are instantly killed without producing offspring  (see Figure \ref{figure1}). 
	\begin{figure*}
		\centering
		\subfigure{
			\begin{minipage}[t]{0.45\textwidth}
				\centering
				\includegraphics[width=8cm]{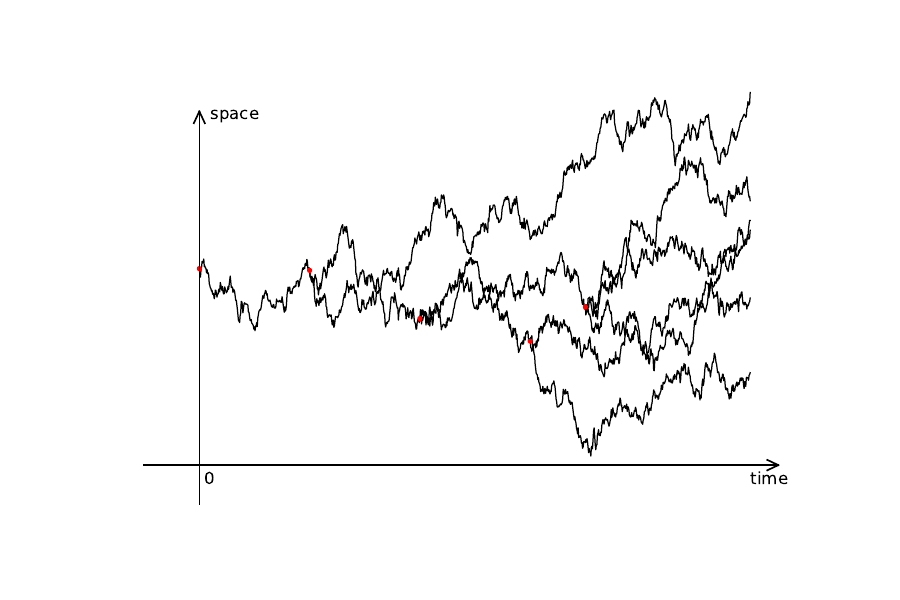}
			\end{minipage}
		}
		\hfill
		\centering
		\subfigure{
			\begin{minipage}[t]{0.45\textwidth}
				\centering
				\includegraphics[width=8cm]{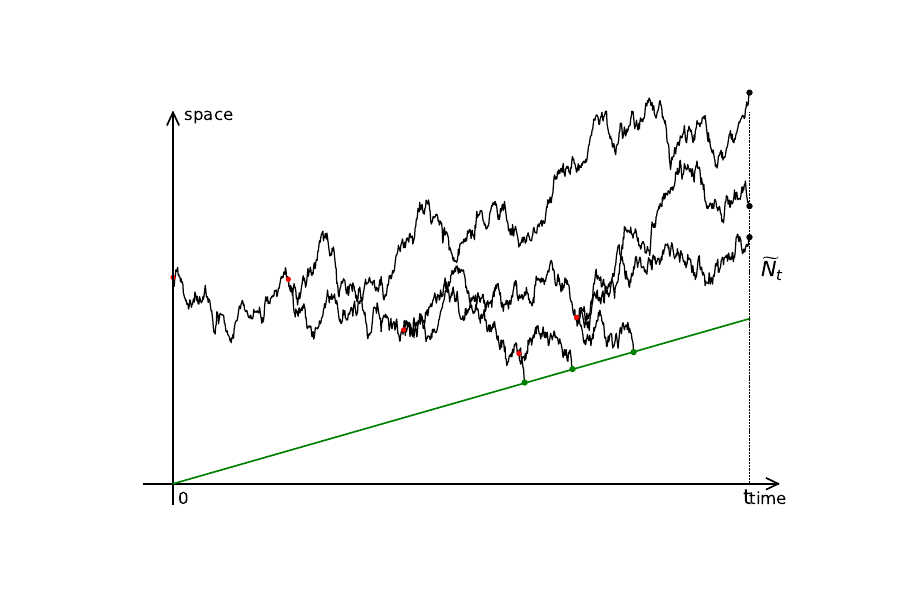}
			\end{minipage}
		}
		\centering
		\caption{BBM and BBM with absorption}
		\label{figure1}
	\end{figure*}
	
	We use $\widetilde{N}_t$ to denote the set of the particles of the BBM with absorption that are still alive at time $t$. For any particle $u\in \widetilde{N}_t$ and any time $s\leq t$, we continue to use $X_u(s)$ to represent the position of either particle $u$ itself or its ancestor at time $s$.  
	The extinction time of the BBM with absorption is defined as
	$$\zeta:=\inf\{t>0: \widetilde{N}_t = \emptyset \}.$$ 
	Additionally, we define $\widetilde{\mathbf{M}}_t$ as the maximum position among all particles $u\in\widetilde{N}_t$. The law of the BBM with absorption, starting from a single particle at position $x$ is denoted by $\mathbb{P}_x$, and its expectation is denoted by $\E_x$. 
	
	The asymptotic behavior of BBM with absorption has been extensively studied in the literature.
	Kesten \cite{Kesten78} demonstrated that the process dies out almost surely when $\rho \ge \sqrt{2}$ while there is a positive probability of survival, that is, $\mathbb{P}_x(\zeta=\infty)>0$, when $\rho < \sqrt{2}$. Therefore, $\rho = \sqrt{2}$ is the critical drift separating the supercritical case $\rho < \sqrt{2}$ and the subcritical $\rho > \sqrt{2}$.
	In the subcritical case,  Harris and Harris \cite{Harris07} provided the large time asymptotic behavior for the survival probability.
	In the critical case, Kesten \cite{Kesten78} obtained upper and lower bounds on the survival probability, which were subsequently improved by Berestycki et al. \cite{Berestycki14}. Maillard and Schweinsberg \cite{Maillard22} have further enhanced these results and investigated the behavior conditioned to survive. For the BBM with absorption in the near-critical case, Berestycki et al. \cite{Berestycki11} and Liu \cite{Liu21} are good references.
	In the supercritical case, Harris et al. \cite{Harris06} studied properties of the right-most particle and the one-sided F-KPP traveling wave solution using probabilistic methods in the case of binary branching. Specifically, they proved that
	\begin{equation}\label{growth rate}
		\lim_{t\to \infty} \frac{\widetilde{\mathbf{M}}_t}{t}=\sqrt{2}\quad on \ \{\zeta=\infty\}, \ \mathbb{P}_x\mbox{-a.s.}
	\end{equation}
	and $g(x) := \mathbb{P}_x(\zeta<\infty)$ is the unique solution to the one-side F-KPP traveling wave solution
	\begin{equation}
		\left \{
		\begin{aligned}
			& \frac{1}{2}g''-\rho g'+ g^2 -g = 0,\quad x>0, \\
			& g(0+)=1,\quad g(\infty)=0.\\
		\end{aligned}\right.
	\end{equation}
	Louidor and Saglietti \cite{Louidor20} showed that the number of particles inside any fixed set, normalized by the mean population size, converges to an explicit limit almost surely.
	In this paper, we focus on the supercritical case. \textit{In the remainder of this paper, we always assume $\rho<\sqrt{2}$.}
	
	In \cite{YZ23}, we studied the maximal displacement and the extremal proccess of BBM with absorption. More precisely, we established the following result:
	\begin{equation}
		\lim_{t\to \infty} \mathbb{P}_x(\widetilde{\mathbf{M}}_t - m_t\leq z) = \E_x(e^{-C_*\widetilde{Z}_{\infty} e^{-\sqrt{2}z} }),
	\end{equation}
	where $\Z_{\infty}$ is defined as the limit of $\Z_t$ and 
	\begin{equation}\label{def_Z}
		\Z_t := \sum_{u\in \widetilde{N}_t} (\sqrt{2}t-X_u(t)) e^{\sqrt{2}(X_u(t)-\sqrt{2}t) }.
	\end{equation}
	It's important to note that $\{\Z_t, t\geq 0, \mathbb{P}_x \}$ is not a martingale. However, according to \cite[Theorem 2.1]{YZ23}, the limit $\Z_{\infty} := \lim_{t\rightarrow\infty} \Z_t$ exists $\mathbb{P}_x$-almost surely for any $x>0$ and $\rho < \sqrt{2}$. 
	Similar to \eqref{ergodic_BBM}, our paper focuses on the empirical distribution function of the maximum of branching Brownian motion with absorption. We prove that the limit of this empirical distribution converges almost surely to a Gumbel distribution with a random shift. Here's the statement of the main result.
	\begin{theorem}[Ergodic Theorem]\label{theorem_ergodic}
		For any $x>0$, $\rho<\sqrt{2}$ and $z\in\R$, we have
		\begin{equation}\label{equation_ergodic}
			\lim_{T\uparrow\infty} \frac{1}{T} \int_0^T \mathbf{1}_{\left\{\widetilde{\mathbf{M}}_t -m_t \leq z\right\}} \mathrm{d}t = \exp\left\{-C_* \Z_{\infty} e^{-\sqrt{2} z} \right\}, \quad \mathbb{P}_x\mbox{-a.s.},
		\end{equation}
		where the positive constant $C_*$ is given by \eqref{travelling}. 
	\end{theorem}

	\section{Proof of Theorem \ref{theorem_ergodic}}\label{Sec:proof_thm}
	We can put a BBM and a BBM with absorption in the same probability space. More precisely, we can construct a BBM as described in Section \ref{Sec:intro}. By considering only the particles that are never killed by the line $\{(y,t): y=\rho t \}$, we obtain a BBM with absorption. Therefore, 
	\begin{equation}\label{def_widetilde_N_t}
		\widetilde{N}_t = \{u\in N_t: \forall s\leq t, X_u(s) > \rho s \}.
	\end{equation}
	\begin{figure*}
		\centering
		\includegraphics[width=10cm]{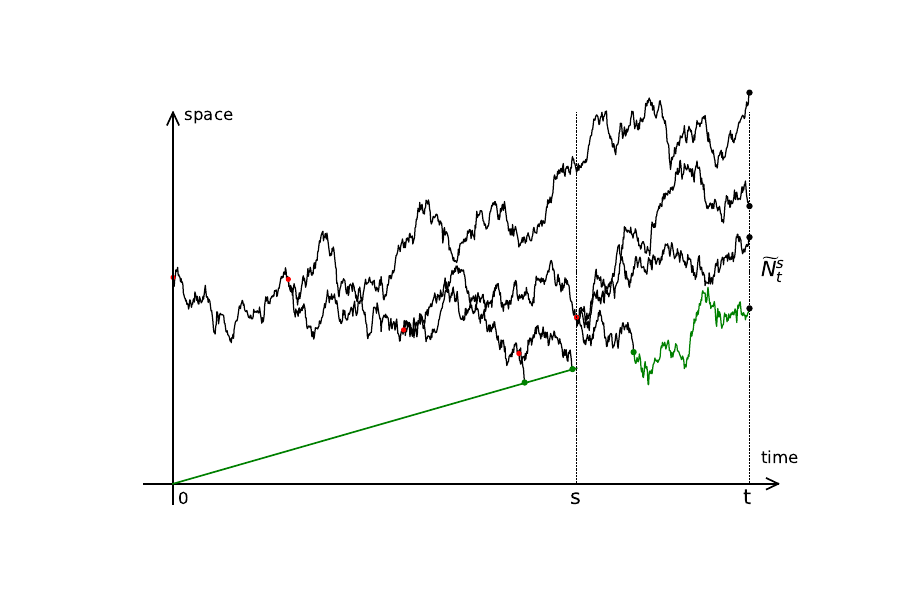}
		\caption{The truncated absorption barrier}
		\label{figure2}
	\end{figure*}	
	Furthermore, for $s\leq t$, we define
	\begin{equation}\label{def_widetilde_N_t^s}
		\widetilde{N}_t^s := \{u\in N_t: \exists v\in \widetilde{N}_s, \mbox{ s.t. } u>v \},
	\end{equation}
	where $u>v$ indicates that $u$ is a descendant of $v$ (see Figure \ref{figure2}). Notice that the set $\widetilde{N}^s_t$ contains all the particles alive at time $t$ that do not hit the line segment $\{(y,r):y =\rho r, 0\leq r\leq s \}$. 
	For convenience, define
	\begin{equation}
		M_t := \max\{X_u(t): u\in N_t \} - m_t \quad \mbox{and} \quad \widetilde{M}_t := \max\{X_u(t): u\in \widetilde{N}_t \} - m_t.
	\end{equation} 
	Then $M_t = \mathbf{M}_t - m_t$ and $\widetilde{M}_t = \widetilde{\mathbf{M}}_t - m_t$. Similarly, define
	\begin{equation}\label{def_widetilde_M_t^s}
		\widetilde{M}_t^s := \max\{X_u(t): u\in \widetilde{N}^s_t\} - m_t.
	\end{equation}
	
	In the proof of Theorem \ref{theorem_ergodic}, we need the following two lemmas, whose proofs are postponed to Sections \ref{Sec:truncated} and \ref{Sec:M_RT_t}. First, as in \cite{ABK13b}, we consider the time $R_T>0$. However, in this paper, we require that there exists some $l>0$ such that $R_T/T^{l}\uparrow \infty$ and $R_T/\sqrt{T}\downarrow 0$, as $T\uparrow\infty$. We truncate the absorption barrier at time $R_T$ and will show that the empirical distribution of $\widetilde{M}_t^{R_T}$ converges almost surely.
	
	\begin{lemma}\label{lemma_Mt_truncated}
		Let $R_T/T^l \uparrow\infty$ as $T\uparrow\infty$ for some $l>0$ but with $R_T = o(\sqrt{T})$. Then for any $x>0$ and $z\in\R$,
		\begin{equation}\label{ergodic_truncated}
			\lim_{T\uparrow\infty} \frac{1}{T} \int_0^T \mathbf{1}_{\left\{ \widetilde{M}_t^{R_T}\leq z\right\}} \mathrm{d}t = \exp\left\{ -C_* \Z_{\infty} e^{-\sqrt{2}z} \right\}, \quad \mathbb{P}_x\mbox{-a.s.}
		\end{equation}
	\end{lemma}
	
	\begin{remark}
	    In \cite{YZ23}, the absorption barrier was truncated at time $s$ and the set $\widetilde{N}_t^s$ was used to approximate $\widetilde{N}_t$. Following the similar approach, we can show that
	    \begin{equation}
	    	\lim_{T\uparrow\infty} \frac{1}{T} \int_0^T \mathbf{1}_{\left\{ \widetilde{M}_t^{s}\leq z\right\}} \mathrm{d}t = \exp\left\{ -C_* \Z_{\infty}^s e^{-\sqrt{2}z} \right\}, \quad \mathbb{P}_x\mbox{-a.s.},
	    \end{equation} 
	    where $\Z_{\infty}^s = \lim\limits_{t\uparrow\infty} \sum_{u\in \widetilde{N}_t^s} (\sqrt{2}t-X_u(t)) e^{\sqrt{2}(X_u(t)-\sqrt{2}t) }$ and $\lim\limits_{s\uparrow\infty} \Z_{\infty}^s = \Z_{\infty}$. Consequently, 
	    \begin{equation}
	    	\lim_{s\uparrow\infty}\lim_{T\uparrow\infty} \frac{1}{T} \int_0^T \mathbf{1}_{\left\{ \widetilde{M}_t^{s}\leq z\right\}} \mathrm{d}t = \exp\left\{ -C_* \Z_{\infty} e^{-\sqrt{2}z} \right\}, \quad \mathbb{P}_x\mbox{-a.s.}
	    \end{equation}
	    However, establishing Theorem \ref{theorem_ergodic} by interchanging the order of limits presents difficulties. To circumvent this, we replace the fixed truncation time $s$ with a $T$-dependent quantity $R_T$. Since $\widetilde{M}_t \leq \widetilde{M}^{R_T}_t$, it follows that
	    \begin{equation}
	    	\liminf_{T\uparrow\infty} \frac{1}{T} \int_0^T \mathbf{1}_{\left\{ \widetilde{M}_t\leq z\right\}} \mathrm{d}t \geq \exp\left\{ -C_* \Z_{\infty} e^{-\sqrt{2}z} \right\}, \quad \mathbb{P}_x\mbox{-a.s.}
	    \end{equation}
	    The lemma above thus establishes one direction of equation \eqref{equation_ergodic}.
	\end{remark}
		
	Next, we consider the difference between $\widetilde{M}_t$ and $\widetilde{M}_t^{R_T}$. For $s<t$, define $\widetilde{N}_t^{[s,t]} = \widetilde{N}_t^s - \widetilde{N}_t$, which represents the set of particles at time $t$ whose ancestors are not absorbed before time $s$ but hit the absorption barrier between time $s$ and $t$. Define
	\begin{equation}\label{def_M_st}
		\widetilde{M}_t^{[s,t]} := \max\left\{ X_u(t): u\in \widetilde{N}_t^{[s,t]} \right\} - m_t.
	\end{equation}
	The following lemma provides the convergence of the empirical distribution of $\widetilde{M}_t^{[s,t]}$.
	\begin{lemma}\label{lemma_M_RT_t}
		For $\epsilon>0$, and $R_T$ as in Lemma \ref{lemma_Mt_truncated},
		\begin{equation}\label{ergodic_M_RT_t}
			\limsup_{T\uparrow\infty} \frac{1}{T} \int_{\epsilon T}^T \mathbf{1}_{\left\{ \widetilde{M}_t^{[R_T,t]} > z\right\}} \mathrm{d}t = 0, \quad \mathbb{P}_x\mbox{-a.s.}
		\end{equation}
	\end{lemma}
	
	\begin{remark}
	    This lemma shows that the particles whose ancestors were absorbed during the interval $[R_T,t]$ do not contribute to the maximum. Moreover, it serves as the key to establishing the converse direction of equation \eqref{equation_ergodic}.
	\end{remark}
	
	\begin{proof}[Proof of Theorem \ref{theorem_ergodic}]
		Note that $\widetilde{N}_t \subset \widetilde{N}_t^{R_T}$. Hence we have $\widetilde{M}_t \leq \widetilde{M}^{R_T}_t$ and then $\mathbf{1}_{\left\{ \widetilde{M}_t \leq z \right\}} \geq \mathbf{1}_{\left\{ \widetilde{M}^{R_T}_t \leq z \right\}}$. By Lemma \ref{lemma_Mt_truncated}, it holds that
		\begin{equation}
			\liminf_{T\uparrow\infty} \frac{1}{T} \int_0^T \mathbf{1}_{\left\{ \widetilde{M}_t\leq z\right\}} \mathrm{d}t \geq \exp\left\{ -C_* \Z_{\infty} e^{-\sqrt{2}z} \right\}, \quad \mathbb{P}_x\mbox{-a.s.}
		\end{equation}
		Since $\mathbf{1}_{\left\{ \widetilde{M}_t\leq z\right\}} + \mathbf{1}_{\left\{ \widetilde{M}_t > z\right\}} = 1$, we have
		\begin{equation}
			\limsup_{T\uparrow\infty} \frac{1}{T} \int_0^T \mathbf{1}_{\left\{ \widetilde{M}_t > z\right\}} \mathrm{d}t \leq 1-\exp\left\{ -C_* \Z_{\infty} e^{-\sqrt{2}z} \right\}, \quad \mathbb{P}_x\mbox{-a.s.}
		\end{equation}
		To prove Theorem \ref{theorem_ergodic}, it suffices to show that
		\begin{equation}\label{Mt_liminf}
			\liminf_{T\uparrow\infty} \frac{1}{T} \int_0^T \mathbf{1}_{\left\{ \widetilde{M}_t > z\right\}} \mathrm{d}t \geq 1-\exp\left\{ -C_* \Z_{\infty} e^{-\sqrt{2}z} \right\}, \quad \mathbb{P}_x\mbox{-a.s.}
		\end{equation} 
		Since $\widetilde{M}_t^s = \max\left\{\widetilde{M}_t, \widetilde{M}_t^{[s,t]}\right\}$, we have
		\begin{equation}
			\mathbf{1}_{\left\{ \widetilde{M}_t^s > z\right\}} \leq \mathbf{1}_{\left\{ \widetilde{M}_t > z\right\}} + \mathbf{1}_{\left\{ \widetilde{M}_t^{[s,t]} > z\right\}},
		\end{equation}
		that is,
		\begin{equation}
			\mathbf{1}_{\left\{ \widetilde{M}_t > z\right\}} \geq \mathbf{1}_{\left\{ \widetilde{M}_t^s > z\right\}} - \mathbf{1}_{\left\{ \widetilde{M}_t^{[s,t]} > z\right\}}.
		\end{equation}
		For $s=R_T$, we have
		\begin{equation}\label{ergodic_liminf_lowerbound}
			\liminf_{T\uparrow\infty} \frac{1}{T} \int_{\epsilon T}^T \mathbf{1}_{\left\{ \widetilde{M}_t > z\right\}} \mathrm{d}t \geq \liminf_{T\uparrow\infty} \frac{1}{T} \int_{\epsilon T}^T \mathbf{1}_{\left\{ \widetilde{M}_t^{R_T} > z\right\}} \mathrm{d}t - \limsup_{T\uparrow\infty} \frac{1}{T} \int_{\epsilon T}^T \mathbf{1}_{\left\{ \widetilde{M}_t^{[R_T,t]} > z\right\}} \mathrm{d}t.
		\end{equation}
		For any $\epsilon>0$, by \eqref{ergodic_liminf_lowerbound} and Lemma \ref{lemma_M_RT_t}, we get that
		\begin{align}
			\liminf_{T\uparrow\infty} \frac{1}{T} \int_0^T \mathbf{1}_{\left\{ \widetilde{M}_t > z\right\}} \mathrm{d}t \geq \liminf_{T\uparrow\infty} \frac{1}{T} \int_{\epsilon T}^T \mathbf{1}_{\left\{ \widetilde{M}_t > z\right\}} \mathrm{d}t
			\geq \liminf_{T\uparrow\infty} \frac{1}{T} \int_{\epsilon T}^T \mathbf{1}_{\left\{ \widetilde{M}_t^{R_T} > z\right\}} \mathrm{d}t.
		\end{align}
		Therefore,
		\begin{align}
			\liminf_{T\uparrow\infty} &\frac{1}{T} \int_0^T \mathbf{1}_{\left\{ \widetilde{M}_t > z\right\}} \mathrm{d}t \geq \liminf_{T\uparrow\infty} \left(\frac{1}{T} \int_{0}^T \mathbf{1}_{\left\{ \widetilde{M}_t^{R_T} > z\right\}} \mathrm{d}t - \frac{1}{T} \int_{0}^{\epsilon T} \mathbf{1}_{\left\{ \widetilde{M}_t^{R_T} > z\right\}} \mathrm{d}t \right)\\
			&\geq\liminf_{T\uparrow\infty} \frac{1}{T} \int_{0}^T \mathbf{1}_{\left\{ \widetilde{M}_t^{R_T} > z\right\}} \mathrm{d}t - \epsilon \geq 1 - \exp\left\{ -C_* \Z_{\infty} e^{-\sqrt{2}z} \right\} - \epsilon,
		\end{align}
		where the last inequality follows from Lemma \ref{lemma_Mt_truncated}.
		Let $\epsilon \downarrow 0$, then the inequality \eqref{Mt_liminf} holds. This completes the proof.
	\end{proof}
	
	\section{Proof of Lemma \ref{lemma_Mt_truncated}}\label{Sec:truncated}
	In this section, we prove Lemma \ref{lemma_Mt_truncated}, which states that for any $x>0$ and $z\in\mathbb{R}$, 
	\begin{equation}
		\lim_{T\uparrow\infty} \frac{1}{T} \int_0^T \mathbf{1}_{\left\{ \widetilde{M}_t^{R_T}\leq z\right\}} \mathrm{d}t = \exp\left\{ -C_* \Z_{\infty} e^{-\sqrt{2}z} \right\}, \quad \mathbb{P}_x\mbox{-a.s.}
	\end{equation}
    Let $\D = [d,D]$ with $-\infty<d<D<\infty$ which is a compact set. Similar to \cite{ABK13b}, this lemma follows from the following two lemmas.
	
	\begin{lemma}\label{lemma_condi}
		For $\epsilon>0$, and $R_T$ as in Lemma \ref{lemma_Mt_truncated}. Then for any $s\in [\epsilon,1]$,
		\begin{equation}
			\lim_{T\uparrow\infty} \mathbb{P}_x\left[ \widetilde{M}^{R_T}_{T\cdot s} \in \D \,|\, \F_{R_T} \right] = \int_{\D} \mathrm{d}\left( \exp\left\{-C_*\Z_{\infty} e^{-\sqrt{2}z} \right\} \right),\quad \mathbb{P}_x\mbox{-a.s.}
		\end{equation}
	\end{lemma}
	
	\begin{lemma}\label{lemma_rest}
		For $\epsilon>0$, and $R_T$ as in Lemma \ref{lemma_Mt_truncated},
		\begin{equation}
			\lim_{T\uparrow\infty} \frac{1}{T} \int_{\epsilon T}^T \left( \mathbf{1}_{\left\{ \widetilde{M}^{R_T}_t \in \D \right\}} - \mathbb{P}_x \left[ \widetilde{M}^{R_T}_{t} \in \D \,|\, \F_{R_T} \right] \right) \mathrm{d}t = 0, \quad \mathbb{P}_x\mbox{-a.s.}
		\end{equation}
	\end{lemma}
	\begin{proof}[Proof of Lemma \ref{lemma_condi}]
		First, we write
		\begin{equation}
			\mathbb{P}_x\left[ \widetilde{M}^{R_T}_{T\cdot s} \in \D \,|\, \F_{R_T} \right] = \mathbb{P}_x\left[ \widetilde{M}^{R_T}_{T\cdot s} \leq D \,|\, \F_{R_T} \right] - \mathbb{P}_x\left[ \widetilde{M}^{R_T}_{T\cdot s} \leq d \,|\, \F_{R_T} \right].
		\end{equation}
		We only need to show the almost surely convergence of the first term. Recall that the definitions \eqref{def_widetilde_N_t}, \eqref{def_widetilde_N_t^s} and \eqref{def_widetilde_M_t^s}, then
		\begin{align}
			\mathbb{P}_x \left[ \widetilde{M}^{R_T}_{T\cdot s} \leq D  \,|\, \F_{R_T} \right] &= \prod_{u\in \widetilde{N}_{R_T}} \mathbb{P}_x \left[ X_u(R_T) + M_{T\cdot s - R_T}(u) + m_{T\cdot s - R_T} \leq D + m_{T\cdot s}  \,|\, \F_{R_T} \right]\\
			= \prod_{u\in \widetilde{N}_{R_T}}& \left(1 - \mathbb{P}_x \left[ M_{T\cdot s - R_T}(u) > D - X_u(R_T) + \sqrt{2}R_T + o_T(1) \,|\, \F_{R_T} \right]\right)
		\end{align}
		where given $\F_{R_T}$, $\{M_{T\cdot s - R_T}(u), u\in \widetilde{N}_{R_T}\}$ are independent and have the same distribution as $\{M_{T\cdot s - R_T}, \mathbb{P}_0\}$. Note that $o_T(1)\rightarrow 0$ as $T\rightarrow \infty$. After time $R_T$, there is no absorption barrier when $\widetilde{M}^{R_T}_{T\cdot s}$ is considered. Therefore, we can use a similar argument as in \cite{ABK13b}. By \cite[equation (20) on p. 1055]{Lalley87}, 
		\begin{equation}\label{min_R_T}
			\lim_{R_T\rightarrow\infty} \min_{u\in N_{R_T}} (\sqrt{2}R_T - X_u(R_T)) = + \infty \quad \mbox{ a.s. }
		\end{equation}
		Let $f(D, R_T) := D - X_u(R_T) + \sqrt{2}R_T + o_T(1)$. By \cite[Lemma 4]{ABK13b}, we have
		\begin{align}
			\mathbb{P}_x \left[ M_{T\cdot s - R_T}(u) > f(D, R_T) \,|\, \F_{R_T} \right]
			= C_* (1+o_r(1)) (1+o_T(1))f(D, R_T) e^{-\sqrt{2} f(D, R_T) },
		\end{align}
		and by \eqref{min_R_T}, this probability tends to zero uniformly for $u\in\widetilde{N}_{R_T}$ as $T\rightarrow\infty$. 
		Hence, 
		\begin{align}
			\mathbb{P}_x \left[ \widetilde{M}^{R_T}_{T\cdot s} \leq D  \,|\, \F_{R_T} \right]
			&= \exp\bigg{(} \sum_{u\in \widetilde{N}_{R_T}} \log \left(1 - \mathbb{P}_x \left[ M_{T\cdot s - R_T}(u) > f(D, R_T) \,|\, \F_{R_T} \right]\right) \bigg{)}\\
			&= \exp\bigg{(} - \sum_{u\in \widetilde{N}_{R_T}} C_* (1+o_r(1)) (1+o_T(1))^2f(D, R_T) e^{-\sqrt{2} f(D, R_T) } \bigg{)}
		\end{align}
	    By the results on the additive martingale of BBM and the ``derivative martingale" of BBM with absorption given in \cite[Theorem 1(ii)]{Kyprianou04} and \cite[Theorem 2.1]{YZ23}, respectively, 
		we have that
		\begin{equation}
			\lim_{T\uparrow\infty}\sum_{u\in \widetilde{N}_{R_T}} D e^{-\sqrt{2}(\sqrt{2}R_T-X_u(R_T))} \leq \lim_{T\uparrow\infty}\sum_{u\in N_{R_T}} D e^{-\sqrt{2}(\sqrt{2}R_T-X_u(R_T))} = 0
		\end{equation}
		and
		\begin{equation}
			\lim_{T\uparrow\infty}\sum_{u\in \widetilde{N}_{R_T}} (\sqrt{2}R_T-X_u(R_T)) e^{-\sqrt{2}(\sqrt{2}R_T-X_u(R_T))} = \Z_{\infty}.
		\end{equation}
		Therefore, we get 
		\begin{align}
			\lim_{T\uparrow\infty} \mathbb{P}_x \left[ \widetilde{M}^{R_T}_{T\cdot s} \leq D  \,|\, \F_{R_T} \right] &= \exp\bigg{(} -  C_* (1+o_r(1)) \lim_{T\uparrow\infty} \sum_{u\in \widetilde{N}_{R_T}} f(D, R_T) e^{-\sqrt{2} f(D, R_T) } \bigg{)}\\
			&= \exp\left( -  C_* (1+o_r(1)) e^{-\sqrt{2}D} \Z_{\infty} \right).
		\end{align}
		Letting $r\uparrow\infty$ yields that 
		\begin{equation}
			\lim_{T\uparrow\infty} \mathbb{P}_x \left[ \widetilde{M}^{R_T}_{T\cdot s} \leq D  \,|\, \F_{R_T} \right] = \exp\left( -  C_* e^{-\sqrt{2}D} \Z_{\infty} \right).
		\end{equation}
		This completes the proof of Lemma \ref{lemma_condi}.
	\end{proof}
	
	\begin{proof}[Proof of Lemma \ref{lemma_rest}]
		Since after time $R_T$, the particle behaves as a branching Brownian motion, most of the proof of \cite[Theorem 3]{ABK13b} is valid for Lemma \ref{lemma_rest}. Now we provide the details.
		
		Fix $x>0$. For $\gamma>0$, $0\leq s \leq t$, define
		\begin{equation}
			F_{\gamma,t}(s) := x + \frac{s}{t}m_t - \min\{s^{\gamma}, (t-s)^{\gamma} \}. 
		\end{equation}
		Choose $0<\alpha<1/2<\beta<1$. We say that a particle $u\in N_t$ is \textit{localized} in the time $t$-tube during the interval $(r,t-r)$ if and only if
		\begin{equation}
			F_{\beta,t}(s) \leq X_u(s) \leq F_{\alpha,t}(s), \, \forall s\in (r,t-r).
		\end{equation}
		Otherwise, we say that it is \textit{not localized}. By \cite[Proposition 6]{ABK13b}, for given $\D=[d,D]$. There exist $r_0, \delta>0$ depending on $\alpha,\beta$ and $\D$ such that for $r\geq r_0$
		\begin{equation}\label{local_estimate}
			\sup_{t\geq 3r} \mathbb{P}_x \left[ \exists u\in N_t: X_u(t) - m_t \in \D \mbox{ but $u$ \textit{not localized} during } (r,t-r) \right] \leq \exp\{-r^{\delta} \}.
		\end{equation}
		Choose $r_T = (20\ln T)^{1/\delta}$. For any $t\in (R_T,T)$, define
		\begin{equation}
			\widetilde{M}_{t,loc}^{R_T} := \max\left\{X_u(t): u\in \widetilde{N}_t^{R_T}, u \textit{ localized} \mbox{ during } (r_T,t-r_T) \right\} - m_t.
		\end{equation}
		Then, we have
		\begin{align}
			\mathbb{P}_x &\left( \left\{\widetilde{M}_{t}^{R_T} \in \D\right\} \setminus \left\{ \widetilde{M}_{t,loc}^{R_T} \in \D \right\} \right)\\ &\leq \mathbb{P}_x \left[ \exists u\in \widetilde{N}_t^{R_T}: X_u(t) - m_t \in \D \mbox{ but $u$ \textit{not localized} during } (r_T,t-r_T) \right]\\
			&\leq \mathbb{P}_x \left[ \exists u\in N_t: X_u(t) - m_t \in \D \mbox{ but $u$ \textit{not localized} during } (r_T,t-r_T) \right].
		\end{align}
		Hence, \eqref{local_estimate} implies that
		\begin{equation}
			\mathbb{P}_x \left( \left\{\widetilde{M}_{t}^{R_T} \in \D\right\} \setminus \left\{ \widetilde{M}_{t,loc}^{R_T} \in \D \right\} \right) \leq \frac{1}{T^{20}}.
		\end{equation}
		Since $\{\widetilde{M}_{t}^{R_T}\geq D \} \cap \{\widetilde{M}_{t,loc}^{R_T} \in \D \} \neq \emptyset$, the inequality $\mathbb{P}_x (\widetilde{M}_{t}^{R_T} \in \D) - \mathbb{P}_x (\widetilde{M}_{t,loc}^{R_T} \in \D) \geq 0$
		is not true. The same issue exists in \cite[(4.11)]{ABK13b}. However, we have the following claim with a slight modification from their paper \cite[Theorems 2.3 and 2.5]{ABK11}.\\
		\textbf{Claim A}: There exist $r_0, \delta>0$ depending on $\alpha,\beta$ and $D$ such that for $r\geq r_0$
		\begin{equation}\label{local_estimate2}
			\sup_{t\geq 3r} \mathbb{P}_x \left[ \exists u\in N_t: X_u(t) - m_t \geq D \mbox{ but $u$ \textit{not localized} during } (r,t-r) \right] \leq \exp\{-r^{\delta} \}.
		\end{equation}
		We prove this claim in Appendix \ref{appendix}.
		Therefore,
		\begin{align}
			\mathbb{P}_x &\left( \left\{\widetilde{M}_{t,loc}^{R_T} \in \D\right\} \setminus \left\{ \widetilde{M}_{t}^{R_T} \in \D \right\} \right)\\ 
			&\leq \mathbb{P}_x \left[ \exists u\in N_t: X_u(t) - m_t \geq D \mbox{ but $u$ \textit{not localized} during } (r_T,t-r_T) \right] \leq \frac{1}{T^{20}}.
		\end{align}
		
		Let
		\begin{equation}
			\text{Rest}_{\epsilon,\D}(T) := \frac{1}{T} \int_{\epsilon T}^T \left( \mathbf{1}_{\left\{ \widetilde{M}^{R_T}_t \in \D \right\}} - \mathbb{P}_x \left[ \widetilde{M}^{R_T}_{t} \in \D \,|\, \F_{R_T} \right] \right) \mathrm{d}t,
		\end{equation}
		and
		\begin{equation}
			\text{Rest}_{\epsilon,\D}^{loc}(T) := \frac{1}{T} \int_{\epsilon T}^T \left( \mathbf{1}_{\left\{ \widetilde{M}^{R_T}_{t,loc} \in \D \right\}} - \mathbb{P}_x \left[ \widetilde{M}^{R_T}_{t,loc} \in \D \,|\, \F_{R_T} \right] \right) \mathrm{d}t.
		\end{equation}
		Using an argument similar to that in the proof of \cite[Lemma 7]{ABK13b}, it holds that
		\begin{equation}\label{rest_restloc}
			\lim_{T\uparrow\infty} \left( \text{Rest}_{\epsilon,\D}(T) - \text{Rest}_{\epsilon,\D}^{loc}(T) \right) = 0, \quad \mathbb{P}_x\mbox{-a.s.}
		\end{equation}
		For completeness, we give the details here. We have 
		\begin{align}
			\text{Rest}_{\epsilon,\D}(T) - \text{Rest}_{\epsilon,\D}^{loc}(T) =&\, \frac{1}{T} \int_{\epsilon T}^T \left( \mathbf{1}_{\left\{ \widetilde{M}^{R_T}_t \in \D \right\}} - \mathbf{1}_{\left\{ \widetilde{M}^{R_T}_{t,loc} \in \D \right\}} \right) \mathrm{d}t\\
			&\,- \frac{1}{T} \int_{\epsilon T}^T \left( \mathbb{P}_x \left[ \widetilde{M}^{R_T}_{t} \in \D \,|\, \F_{R_T} \right] -  \mathbb{P}_x \left[ \widetilde{M}^{R_T}_{t,loc} \in \D \,|\, \F_{R_T} \right] \right) \mathrm{d}t\\
			=&: \mathbf{(1)_{T,\epsilon}} - \mathbf{(2)_{T,\epsilon}}
		\end{align}
		Notice that
		\begin{align}
			&\E_x \left| \frac{1}{T} \int_{\epsilon T}^T \left( \mathbf{1}_{\left\{ \widetilde{M}^{R_T}_t \in \D \right\}} - \mathbf{1}_{\left\{ \widetilde{M}^{R_T}_{t,loc} \in \D \right\}} \right) \mathrm{d}t \right| \leq 
			\frac{1}{T} \int_{\epsilon T}^T \E_x \left| \mathbf{1}_{\left\{ \widetilde{M}^{R_T}_t \in \D \right\}} - \mathbf{1}_{\left\{ \widetilde{M}^{R_T}_{t,loc} \in \D \right\}} \right| \mathrm{d}t\\
			&\leq \frac{1}{T} \int_{\epsilon T}^T \left[\mathbb{P}_x \left( \left\{\widetilde{M}_{t}^{R_T} \in \D\right\} \setminus \left\{ \widetilde{M}_{t,loc}^{R_T} \in \D \right\} \right) + \mathbb{P}_x \left( \left\{\widetilde{M}_{t,loc}^{R_T} \in \D\right\} \setminus \left\{ \widetilde{M}_{t}^{R_T} \in \D \right\} \right)  \right] \mathrm{d}t\\
			&\leq \frac{2}{T^{20}}.
		\end{align}
		Then, it follows from Markov's inequality that
		\begin{align}
			\mathbb{P}\left( |\mathbf{(1)_{T,\epsilon}}| > \delta \right) \leq \frac{1}{\delta} \E_x \left| \frac{1}{T} \int_{\epsilon T}^T \left( \mathbf{1}_{\left\{ \widetilde{M}^{R_T}_t \in \D \right\}} - \mathbf{1}_{\left\{ \widetilde{M}^{R_T}_{t,loc} \in \D \right\}} \right) \mathrm{d}t \right| \leq \frac{2}{\delta T^{20}},
		\end{align}
		which is summable over $T$ (noticing that we may take $T\in\mathbb{N}$). Hence, by Borel-Cantelli lemma,
		\begin{align}
			\mathbb{P}\left( \{ |\mathbf{(1)_{T,\epsilon}}| > \delta \} \mbox{ infinitely often} \right) = 0.
		\end{align}
		Since this holds for every $\delta>0$, it follows that $\lim_{T\uparrow \infty}  \mathbf{(1)_{T,\epsilon}} = 0$ $\mathbb{P}$-a.s. The proof of $\lim_{T\uparrow \infty}  \mathbf{(2)_{T,\epsilon}} = 0$ is analogous. Equation \eqref{rest_restloc} is therefore established.
		Then, it suffices to prove that $\lim_{T\uparrow\infty} \text{Rest}_{\epsilon,\D}^{loc}(T) = 0$, $\mathbb{P}_x\mbox{-a.s.}$
		
		Define
		\begin{equation}
			X_t^{\{D\}} :=  \mathbf{1}_{\left\{ \widetilde{M}^{R_T}_{t,loc} \leq D \right\}} - \mathbb{P}_x \left[ \widetilde{M}^{R_T}_{t,loc} \leq D \,|\, \F_{R_T} \right] \mbox{ and } \widehat{C}_T(t,t') := \mathbb{E}_x\left[X_t^{\{D\}}\cdot X_{t'}^{\{D\}}\right].
		\end{equation}
		Then, we have
		\begin{align}
			\text{Rest}_{\epsilon,\D}^{loc}(T) = \frac{1}{T} \int_{\epsilon T}^T 	X_s^{\{D\}} \mathrm{d} s - \frac{1}{T} \int_{\epsilon T}^T X_s^{\{d\}} \mathrm{d} s. 
		\end{align} 
		By \cite[Theorem 8]{ABK13b}, we only need to verify that both integrals satisfy the assumptions of this theorem. Once this is done, we can conclude that $\frac{1}{T}\int_{\epsilon T}^T X_s^{\{D\}} \mathrm{d} s$ and $\frac{1}{T} \int_{\epsilon T}^T X_s^{\{d\}} \mathrm{d} s$ converge to $0$ as $T\to\infty$.
		By definition, $|X_s^{\{D\}}|\leq 2$ almost surely for all $s>0$, and $\E[X_s^{\{D\}}] = 0$. Hence, it is enough to check that
		\begin{align}
			\sum_{T=1}^{\infty} \frac{1}{T} \E \left[ \left| \frac{1}{T}\int_{\epsilon T}^T X_s^{\{D\}} \mathrm{d} s \right|^2 \right] < \infty.
		\end{align}
		Following the computations in \cite[equations (4.28), (4.29)]{ABK13b}, we obtain
		\begin{align}
			\sum_{T=1}^{\infty} &\frac{1}{T} \E \left[ \left| \frac{1}{T}\int_{\epsilon T}^T X_s^{\{D\}} \mathrm{d} s \right|^2 \right] = 2\sum_{T=1}^{\infty} \frac{1}{T^3} \int_{\epsilon T}^T \mathrm{d} s \int_s^T \mathrm{d} s' \widehat{C}_T(s,s')\\
			&= 2\sum_{T=1}^{\infty} \frac{1}{T^3} \int_{\epsilon T}^T \mathrm{d} s \int_s^{s+T^{\xi}} \mathrm{d} s' \widehat{C}_T(s,s') + 2\sum_{T=1}^{\infty} \frac{1}{T^3} \int_{\epsilon T}^T \mathrm{d} s \int_{s+T^{\xi}}^T \mathrm{d} s' \widehat{C}_T(s,s'), \label{sum_C_T}
		\end{align}
		where $\xi\in(0,1)$ is a fixed constant. The first term in \eqref{sum_C_T} is bounded by $2 \sum_{T=1}^{\infty} \frac{1}{T^3} T\cdot T^{\xi} \cdot 1$, which is finite. It remains to show that the second term in \eqref{sum_C_T} is also finite. For this, it suffices to verify the validity of \cite[Theorem 9]{ABK13b} for this $\widehat{C}_T(t,t')$. 
		
		Theorem 9 of \cite{ABK13b} provides a key estimate, those proof occupies a substantial portion of the original paper. In our setting, due to the definition of $\widetilde{N}^{R_T}_{t}$, the absorption barrier exists only until time $R_T$. Therefore, the majority of the proof in \cite[Theorem 9]{ABK13b} remains valid for our choice of $\widehat{C}_T(t,t')$. Below we outline only the parts that differ from their argument and explain why the estimate still holds for our $\widehat{C}_T(t,t')$. 
		Define
		\begin{align}\label{def_c_T}
			\hat{c}_T(I,J) :=&\, \mathbb{P}_x \left[ \widetilde{M}^{R_T}_{I,loc} \leq D, \widetilde{M}^{R_T}_{J,loc} \leq D \,|\, \F_{R_T} \right]\\
			&- \mathbb{P}_x \left[ \widetilde{M}^{R_T}_{I,loc} \leq D \,|\, \F_{R_T} \right] \mathbb{P}_x \left[ \widetilde{M}^{R_T}_{J,loc} \leq D \,|\, \F_{R_T} \right],
		\end{align}
		where we use $I$ and $J$ to denote the two times $t,t'$, respectively. Then $\widehat{C}_T(I,J) := \mathbb{E}_x \hat{c}_T(I,J)$. Let 
		\begin{equation}
			I_T:= I-R_T \mbox{ and } J_T:= J-R_T.
		\end{equation}
		Given $X_u(R_T)$, denote by $M^{R_T,u}_{I_T,loc}$ the maximum shifted by $m_{I_T}$ of the positions of all particles $v\in N_{I_T}$ whose paths satisfy
		\begin{align}
			 X_v(s)\geq \sqrt{2}R_T-X_u(R_T) + \frac{s}{I} m_I - \min\{(R_T+s)^{\beta},(I-R_T-s)^{\beta}\}, \\
			 X_v(s) \leq \sqrt{2}R_T-X_u(R_T) + \frac{s}{I} m_I - \min\{(R_T+s)^{\alpha},(I-R_T-s)^{\alpha}\},
		\end{align} 
		for all $s\in [0,I_T-r_T]$, the ``shifted" $I$-tube. Similar to \cite[(5.4)]{ABK13b}, we know that
		\begin{align}
			&\mathbb{P}_x \left[ \widetilde{M}^{R_T}_{I,loc} \leq D, \widetilde{M}^{R_T}_{J,loc} \leq D \,|\, \F_{R_T} \right] \\
			=& \prod_{u\in \widetilde{\star}} \mathbb{P}_x \left( M^{R_T,u}_{I_T,loc} \leq D + \sqrt{2}R_T - X_u(R_T), M^{R_T,u}_{J_T,loc} \leq D + \sqrt{2}R_T - X_u(R_T) | X_u(R_T) \right),
		\end{align}
		where $\widetilde{\star}$ denotes the set of particles $u\in\widetilde{N}_{R_T}$ whose paths are \textit{localized} in the intersection of the $I$-tube and $J$-tube during the interval $(r_T,R_T)$. Because the required control can be reduced to restrictions on the localized positions at time $R_T$, we introduce the set
		\begin{equation}
			\widetilde{\Delta} := \left\{u\in\widetilde{N}_{R_T}: \sqrt{2}R_T-X_u(R_T) \in (R_T^{\alpha}+\Omega_T,R_T^{\beta}+\Omega_T) \right\},
		\end{equation} 
		where $\Omega_T$ denotes a term that is negligible in the relevant asymptotics (its precise form may vary from one occurrence to another). Then, we obtain that
		\begin{align}
			&\mathbb{P}_x \left[ \widetilde{M}^{R_T}_{I,loc} \leq D, \widetilde{M}^{R_T}_{J,loc} \leq D \,|\, \F_{R_T} \right] \\
			\leq& \prod_{u\in \widetilde{\Delta}} \mathbb{P}_x \left( M^{R_T,u}_{I_T,loc} \leq D + \sqrt{2}R_T - X_u(R_T), M^{R_T,u}_{J_T,loc} \leq D + \sqrt{2}R_T - X_u(R_T) | X_u(R_T) \right).
		\end{align}

		Compared with $M_{loc}(I)$ in \cite{ABK13b}, our $\widetilde{M}^{R_T}_{I,loc}$ has an additional restriction before time $R_T$, namely, that the particles have not been absorbed. Consequently, $\widetilde{\Delta}$ is a subset of $\Delta$ defined in \cite{ABK13b}.
		Furthermore, \cite[Propostion 10]{ABK13b} remains valid for our $\hat{c}_T(I,J)$, and \cite[(5.11)]{ABK13b} also serves as an upper bound for this $\hat{c}_T(I,J)$. (Although \cite[(5.15)]{ABK13b} should be $-a-a^2\leq \ln(1-a) \leq -a$ $(0\leq a\leq 1/2)$, \cite[(5.11)]{ABK13b} is still true due to the inequality $-a\leq-a+a^2/2$.) The results of \cite{ABK13b} thus imply Lemma \ref{lemma_rest}.
	\end{proof}

	Now we turn to the proof of Lemma \ref{lemma_Mt_truncated} and the argument is similar to that in \cite{ABK13b}.
	\begin{proof}[Proof of Lemma \ref{lemma_Mt_truncated}]
		For any compact interval $\D = [d,D]$ with $-\infty<d<D<\infty$, we will show that
		\begin{equation}\label{ergodic_interval}
			\lim_{T\uparrow\infty} \frac{1}{T} \int_0^T \mathbf{1}_{\left\{ \widetilde{M}_t^{R_T}\in \mathcal{D} \right\}} \mathrm{d}t = \int_{\D} \mathrm{d}\left( \exp\left\{-C_*\Z_{\infty} e^{-\sqrt{2}z} \right\} \right),\quad \mathbb{P}_x\mbox{-a.s.}
		\end{equation}
		Note that
		\begin{align}
			\frac{1}{T} \int_0^T \mathbf{1}_{\left\{ \widetilde{M}_t^{R_T}\in \mathcal{D} \right\}} \mathrm{d}t &= \frac{1}{T} \int_0^{\epsilon T} \mathbf{1}_{\left\{ \widetilde{M}_t^{R_T}\in \mathcal{D} \right\}} \mathrm{d}t + \frac{1}{T} \int_{\epsilon T}^T \mathbf{1}_{\left\{ \widetilde{M}_t^{R_T}\in \mathcal{D} \right\}} \mathrm{d}t\\
			&= \frac{1}{T} \int_0^{\epsilon T} \mathbf{1}_{\left\{ \widetilde{M}_t^{R_T}\in \mathcal{D} \right\}} \mathrm{d}t + \frac{1}{T} \int_{\epsilon T}^T  \mathbb{P}_x \left[ \widetilde{M}^{R_T}_{t} \in \D \,|\, \F_{R_T} \right]  \mathrm{d}t\\ 
			&\quad+ \frac{1}{T} \int_{\epsilon T}^T \left( \mathbf{1}_{\left\{ \widetilde{M}^{R_T}_t \in \D \right\}} - \mathbb{P}_x \left[ \widetilde{M}^{R_T}_{t} \in \D \,|\, \F_{R_T} \right] \right) \mathrm{d}t.
		\end{align}
		By Lemma \ref{lemma_condi} and the dominated convergence theorem, it holds that
		\begin{align}
			\lim_{\epsilon \downarrow 0} & \lim_{T\uparrow\infty} \frac{1}{T} \int_{\epsilon T}^T  \mathbb{P}_x \left[ \widetilde{M}^{R_T}_{t} \in \D \,|\, \F_{R_T} \right]  \mathrm{d}t = \lim_{\epsilon \downarrow 0} \lim_{T\uparrow\infty} \int_{\epsilon}^1 \mathbb{P}_x \left[ \widetilde{M}^{R_T}_{T\cdot s} \in \D \,|\, \F_{R_T} \right]  \mathrm{d}s\\
			&= \lim_{\epsilon \downarrow 0} \int_{\epsilon}^1 \lim_{T\uparrow\infty} \mathbb{P}_x \left[ \widetilde{M}^{R_T}_{T\cdot s} \in \D \,|\, \F_{R_T} \right]  \mathrm{d}s = \int_{\D} \mathrm{d}\left( \exp\left\{-C_*\Z_{\infty} e^{-\sqrt{2}z} \right\} \right),\quad \mathbb{P}_x\mbox{-a.s.}
		\end{align}
		Combining this with Lemma \ref{lemma_rest} and $\lim_{\epsilon \downarrow 0} \lim_{T\uparrow\infty} \frac{1}{T} \int_0^{\epsilon T} \mathbf{1}_{\left\{ \widetilde{M}_t^{R_T}\in \mathcal{D} \right\}} \mathrm{d}t = 0$, we get \eqref{ergodic_interval}. 
		
		Note that
		\begin{equation}
			\lim_{D\rightarrow \infty} \int_{(D,\infty)} \mathrm{d}\left( \exp\left\{-C_*\Z_{\infty} e^{-\sqrt{2}z} \right\} \right) = \lim_{D\rightarrow \infty}\left( 1 - \exp\left\{-C_*\Z_{\infty} e^{-\sqrt{2}D} \right\} \right) = 0
		\end{equation}
		and
		\begin{equation}
			\lim_{D\rightarrow \infty}\lim_{T\uparrow\infty} \frac{1}{T} \int_0^T \mathbf{1}_{\left\{ \widetilde{M}_t^{R_T} > D \right\}} \mathrm{d}t \leq \lim_{D\rightarrow \infty}\lim_{T\uparrow\infty} \frac{1}{T} \int_0^T \mathbf{1}_{\left\{ \widetilde{M}_t > D \right\}} \mathrm{d}t = 0,\quad \mathbb{P}_x\mbox{-a.s.}
		\end{equation} 
		Combining the two displays above with \eqref{ergodic_interval}, we get that
		\begin{equation}
			\lim_{T\uparrow\infty} \frac{1}{T} \int_0^T \mathbf{1}_{\left\{ \widetilde{M}_t^{R_T}\geq d \right\}} \mathrm{d}t = 1 -  \exp\left\{-C_*\Z_{\infty} e^{-\sqrt{2}d} \right\},\quad \mathbb{P}_x\mbox{-a.s.}
		\end{equation}
		Hence \eqref{ergodic_truncated} holds. This completes the proof.
	\end{proof}
	
	\section{Proof of Lemma \ref{lemma_M_RT_t}}\label{Sec:M_RT_t}
	\begin{proof}[Proof of Lemma \ref{lemma_M_RT_t}]
		For any $u\in N_t$, define 
		\begin{equation}
			\underline{\tau}(u) := \inf\{s\in [0,t]: X_u(s) \leq \rho s \}.
		\end{equation}
		and $\inf\emptyset$ is defined by $+\infty$. Then $\underline{\tau}(u)$ represents the time when the particle $u$ or its ancestor first hits the absorption barrier. Then
		\begin{equation}
			\widetilde{N}_t^{[s,t]} = \left\{u\in N_t: \underline{\tau}(u) \in [s,t] \right\}.
		\end{equation}
		Similarly, for $0<s<s'<t$, define
		\begin{equation}
			\widetilde{N}_t^{[s,s']} = \left\{u\in N_t: \underline{\tau}(u) \in [s,s'] \right\}.
		\end{equation}
		And let $\widetilde{M}_t^{[s,s']} = \max\left\{X_u(t): u\in \widetilde{N}_t^{[s,s']} \right\} - m_t$.
		
		Let $p\in (0,1)$. Notice that $\widetilde{N}_t^{[R_T,t]} = \widetilde{N}_t^{[R_T,pt]} \cup \widetilde{N}_t^{[pt,t]}$. Therefore, 
		\begin{equation}
			\widetilde{M}_t^{[R_T,t]} \leq \max\left\{\widetilde{M}_t^{[R_T,pt]}, \, \widetilde{M}_t^{[pt,t]}  \right\},
		\end{equation}
		and moreover
		\begin{equation}\label{M_RT_t_bound}
			\mathbf{1}_{\left\{ \widetilde{M}_t^{[R_T,t]} > z\right\}} \leq \mathbf{1}_{\left\{ \widetilde{M}_t^{[R_T,pt]} > z\right\}} + \mathbf{1}_{\left\{ \widetilde{M}_t^{[pt,t]} > z\right\}}.
		\end{equation}
		The proof of \cite[Lemma 4.3]{YZ23} gave an upper bound for both $\mathbb{P}_x(\widetilde{M}_t^{[pt,t]} > z)$ and $\mathbb{P}_x(\widetilde{M}_t^{[R_T,pt]} > z)$. In \cite[Lemma 4.3]{YZ23}, let $A=-z$, then $I = \mathbb{P}_x(\widetilde{M}_t^{[pt,t]} > z)$ and $II = \mathbb{P}_x(\widetilde{M}_t^{[R_T,pt]} > z)$.
		By the proof of \cite[Lemma 4.3]{YZ23}, we have
		\begin{equation}
			I = 			\mathbb{P}_x\left(\widetilde{M}_t^{[pt,t]} > z\right) \leq \frac{C}{p^{\frac{3}{2}}} \int_{pt}^{\infty} e^{-\left( \frac{\rho}{\sqrt{2}} - 1\right)^2 r } \mathrm{d} r
		\end{equation}
		and
		\begin{align}
			II = \mathbb{P}_x\left(\widetilde{M}_t^{[R_T,pt]} > z\right) &\leq C  \Pi_x\left[\underline{\tau}_{0}^{\sqrt{2}-\rho} \mathbf{1}_{\{\underline{\tau}_{0}^{\sqrt{2}-\rho} \geq R_T \}} \right]\\
			&= C \int_{R_T}^{\infty} \frac{rx}{\sqrt{2\pi r^3}} e^{-\frac{(x-(\sqrt{2}-\rho)r)^2}{2r}} \mathrm{d} r,
		\end{align}
		where the positive constant $C$ changes line by line and depends only on $x,\rho,z$. Here, $\{B_t, \Pi_x \}$ is a standard Brownian motion starting from $x$, and $\underline{\tau}_{0}^{\sqrt{2}-\rho} := \inf\{s\geq 0: B_s \leq (\sqrt{2}-\rho)s \}$.
		
		For any $\delta>0$, by Markov's inequality and Fubini's theorem, 
		\begin{align}
			\mathbb{P}_x\left( \frac{1}{T} \int_{\epsilon T}^T \mathbf{1}_{\left\{ \widetilde{M}_t^{[pt,t]} > z\right\}} \mathrm{d}t \geq \delta \right) &\leq \frac{1}{\delta}\frac{1}{T} \int_{\epsilon T}^T \mathbb{P}\left( \widetilde{M}_t^{[pt,t]} > z\right) \mathrm{d}t\\ 
			&\leq \frac{C}{\delta T p^{3/2}} \int_{\epsilon T}^T \int_{pt}^{\infty} e^{-\left( \frac{\rho}{\sqrt{2}} - 1\right)^2 r } \mathrm{d} r \mathrm{d}t\\
			&\leq \frac{C}{\delta T p^{5/2}} e^{-\left( \frac{\rho}{\sqrt{2}} - 1\right)^2 p\epsilon T},
		\end{align}
		which is summable over $T\in \mathbb{N}$. Therefore, by Borel-Cantelli lemma,
		\begin{equation}
			\mathbb{P}_x\left[ \left\{ \frac{1}{T} \int_{\epsilon T}^T \mathbf{1}_{\left\{ \widetilde{M}_t^{[pt,t]} > z\right\}} \mathrm{d}t \geq \delta  \right\} \mbox{ infinitely often }  \right] = 0.
		\end{equation}
		Hence,
		\begin{equation}\label{ergodic_M_pt_t}
			\limsup_{T\uparrow\infty} \frac{1}{T} \int_{\epsilon T}^T \mathbf{1}_{\left\{ \widetilde{M}_t^{[pt,t]} > z\right\}} \mathrm{d}t = 0, \quad \mathbb{P}_x\mbox{-a.s.}
		\end{equation}
		Using the same argument as above,
		\begin{align}
			\mathbb{P}_x\left( \frac{1}{T} \int_{\epsilon T}^T \mathbf{1}_{\left\{ \widetilde{M}_t^{[R_T,pt]} > z\right\}} \mathrm{d}t \geq \delta \right) &\leq \frac{C}{\delta T} \int_{\epsilon T}^T  \int_{R_T}^{\infty} \frac{rx}{\sqrt{2\pi r^3}} e^{-\frac{(x-(\sqrt{2}-\rho)r)^2}{2r}} \mathrm{d} r \mathrm{d}t\\ 
			&\leq \frac{C}{\delta} \int_{R_T}^{\infty} \frac{x}{\sqrt{2\pi r}} e^{-\frac{(\sqrt{2}-\rho)^2}{2}r + (\sqrt{2}-\rho)x} \mathrm{d} r,
		\end{align}
		which is summable in $T\in\mathbb{N}$ when $R_T/T^l\rightarrow\infty$ as $T\rightarrow\infty$ for some $l>0$. Therefore,
		\begin{equation}\label{ergodic_M_RT_pt}
			\limsup_{T\uparrow\infty} \frac{1}{T} \int_{\epsilon T}^T \mathbf{1}_{\left\{ \widetilde{M}_t^{[R_T,pt]} > z\right\}} \mathrm{d}t = 0, \quad \mathbb{P}_x\mbox{-a.s.}
		\end{equation}
		By \eqref{M_RT_t_bound}, \eqref{ergodic_M_pt_t} and \eqref{ergodic_M_RT_pt}, we get \eqref{ergodic_M_RT_t}. This completes the proof.
	\end{proof}
	
	\appendix
	\section{Proof of Claim A}\label{appendix}
	\begin{proof}
		The results in \cite{ABK11} already imply this claim, and it can be directly derived with only minor modifications.
		According to equations (5.5), (5.54), (5.62) and (5.63) in \cite{ABK11}, \eqref{local_estimate} holds for any compact interval $\mathcal{D}$. However, we require \eqref{local_estimate} to remain valid when the compact interval $\D$ is replaced by $[D,\infty)$. With a slight modification, the proof given in \cite{ABK11} continues to apply.
		
		For any $x\in \R$, $t>r>0$ and interval $\mathcal{A}$, define
		\begin{equation}\label{def_PxD}
			\mathrm{P}\left(x,t,r,\mathcal{A}\right) = \mathbb{P}_x \left[ \exists u\in N_t: X_u(t) - m_t \in \mathcal{A} \mbox{ but $u$ \textit{not localized} during } (r,t-r) \right].
		\end{equation}
		Let $\mathcal{D}_b = [b,b+1]$ and $\theta\in(0,\alpha)$. We write $\lfloor s \rfloor$ for the largest integer less than or equal to $s$. 
		The key proof strategy of this claim is to decompose $[D,\infty)$ as
		\begin{equation}
			[D,\infty) \subset \left(\bigcup_{b=\lfloor D \rfloor}^{\lfloor r^{\theta} \rfloor} \mathcal{D}_b \right) \cup [r^{\theta},\infty),
		\end{equation}
		and then apply the result from \cite{ABK11} to $\mathrm{P}\left(x,t,r,\D_b\right)$.
		First, we will show that there exist $r_1,\delta_1>0$ depending on $\alpha,\beta$ and $D$ such that for $r\geq r_1$, $b\in [\lfloor D \rfloor,r^{\theta}]$, the following holds:
		\begin{equation}\label{local_estimate_b}
			\sup_{t\geq 3r} \mathrm{P}\left(x,t,r,\D_b\right)\leq \exp\{-r^{\delta_1} \}.
		\end{equation}
		To prove \eqref{local_estimate_b}, we repeat the proof of \cite[Theorems 2.3 and 2.5]{ABK11} on pages 1668-1674 when the compact set $\mathcal{D}$ is replaced by $\D_b = [b,b+1]$. 
		
		When $\D_b = [b,b+1]$, we have $\overline{D} = b+1$ and $\underline{D} = b$ in \cite[(5.20)]{ABK11}. Notice that for $b\geq D$, \cite[(5.25)]{ABK11} implies that
		\begin{equation}
			e^t \Pi[B_t\in m_t+\D_b] \leq \kappa t e^{-\sqrt{2}D}.
		\end{equation}
		The upper bound of this probability does not depend on $b$. In \cite[line 5 on page 1670]{ABK11}, we need to choose $r$ large enough so that $\underline{F}\leq 0$ on $[r,t-r]$ where $\underline{F}$ is given by \cite[(5.32)]{ABK11}. For any $\delta,t>0$, define $f_{t,\delta}(s) = \min\{s^{\delta},(t-s)^{\delta}\}$ for $0\leq s \leq t$. For $b<r^{\theta}<r^{\alpha}$ and any $s\in [r,t-r]$,
		\begin{equation}
			\underline{F}(s) = b\frac{t-s}{t} - f_{t,\alpha}(s) \leq b - r^{\alpha} < 0,
		\end{equation}
		which satisfies the condition. The remaining proof of \cite[Theorem 2.3]{ABK11} remains unchanged.
		
		In the proof of \cite[Theorem 2.5]{ABK11}, for sufficiently large $r$, we have
		\begin{equation}
			\mathrm{diam}(\mathcal{D}) = |\overline{D}| + |\underline{D}| \leq 2b+1 \leq 3r^{\theta}.
		\end{equation}
		Notice that $0<\theta<\alpha<\frac{1}{2}<\beta<1$ and let $0<a<1$ such that $2a\beta-1>0$. We can find $\widetilde{r} = \widetilde{r}(\alpha,\beta,\theta,D,a)$ such that for $r\geq \widetilde{r}$ the following holds:
		\begin{equation}
			3r^{\theta} - f_{t,\alpha}(s) \leq 0 \mbox{ and }  3r^{\theta} - f_{t,\beta}(s) \leq -f_{t,a\beta}(s) \mbox{ for all } r\leq s\leq t-r.
		\end{equation}
		Therefore, the proof of \cite[Theorem 2.5]{ABK11} is also valid and \eqref{local_estimate_b} holds for $b\in [D,r^{\theta}]$.
		
		By \cite[Corollary 10]{ABK12}, for sufficiently large $r$ and $t\geq 3r$, we have
		\begin{align}\label{local_estimate_rtheta}
			\mathrm{P}\left(x,t,r,[r^{\theta},\infty)\right)
			\leq  \mathbb{P}_x \left[ M_t \geq m_t + r^{\theta} \right]
			\leq C r^{\theta} e^{-\sqrt{2}r^{\theta}+1},
		\end{align}
		for some constant $C>0$. Therefore, combining \eqref{local_estimate_b} with \eqref{local_estimate_rtheta}, we get that
		\begin{align}
			\sup_{t\geq 3r} \mathrm{P}\left(x,t,r,[D,\infty)\right) 
			\leq &\, \sum_{b = \lfloor D \rfloor}^{\lfloor r^{\theta} \rfloor }  \mathrm{P}(x,t,r,\mathcal{D}_b) + \mathrm{P}\left(x,t,r,[r^{\theta},\infty)\right)\\
			\leq &\, (|D|+r^{\theta}) e^{-r^{\delta_1}} + C r^{\theta} e^{-\sqrt{2}r^{\theta}+1}.
		\end{align}
		Choose $\delta < \min\{\delta_1,\theta\}$, there exists a sufficiently large $r_0$ such that \eqref{local_estimate2} holds for $r\geq r_0$. This completes the proof.
	\end{proof}
	
	\vspace{.1in}
	\textbf{Acknowledgment}:
	We thank Professor Xinxin Chen for many constructive suggestions, and, in particular, for providing the proof strategy of Claim A.

\end{document}